\documentclass[12pt,a4paper]{article}
\usepackage{indentfirst}
\usepackage{amsmath,amssymb,amsfonts,amsthm,graphics}
\usepackage{makeidx}
\usepackage{amsmath,amssymb,amsfonts,amsthm,graphics,graphicx}
\usepackage{makeidx}
\usepackage{color}
\usepackage{ifpdf}
\usepackage{subfigure}

\newtheorem{thm}{Theorem}

\newtheorem{lem}{Lemma}

\theoremstyle{definition}

\def\-{\mbox{--}}

\newtheorem{pro}{Proposition}
\newtheorem{claim}{Claim}

\newtheorem{remark}{Remark}

\begin{document}

\title{\Large\bf Good upper bounds for the\\ total rainbow connection
of graphs\footnote{Supported by NSFC No.11371205 and PCSIRT.} }
\author{\small Hui Jiang, Xueliang Li, Yingying Zhang\\
\small Center for Combinatorics and LPMC-TJKLC\\
\small Nankai University, Tianjin 300071, China\\
\small E-mail: jhuink@163.com; lxl@nankai.edu.cn; zyydlwyx@163.com}
\date{}
\maketitle
\begin{abstract}
A total-colored graph is a graph $G$ such that both all edges and
all vertices of $G$ are colored. A path in a total-colored graph $G$
is a total rainbow path if its edges and internal vertices have
distinct colors. A total-colored graph $G$ is total-rainbow
connected if any two vertices of $G$ are connected by a total
rainbow path of $G$. The total rainbow connection number of $G$,
denoted by $trc(G)$, is defined as the smallest number of colors
that are needed to make $G$ total-rainbow connected. These concepts
were introduced by Liu et al. Notice that for a connected graph $G$,
$2diam(G)-1\leq trc(G)\leq 2n-3$, where $diam(G)$ denotes the
diameter of $G$ and $n$ is the order of $G$. In this paper we show,
for a connected graph $G$ of order $n$ with minimum degree $\delta$,
that $trc(G)\leq6n/{(\delta+1)}+28$ for $\delta\geq\sqrt{n-2}-1$ and
$n\geq 291$, while $trc(G)\leq7n/{(\delta+1)}+32$ for
$16\leq\delta\leq\sqrt{n-2}-2$ and
$trc(G)\leq7n/{(\delta+1)}+4C(\delta)+12$ for $6\leq\delta\leq15$,
where
$C(\delta)=e^{\frac{3\log({\delta}^3+2{\delta}^2+3)-3(\log3-1)}{\delta-3}}-2$.
This implies that when $\delta$ is in linear with $n$, then the
total rainbow number $trc(G)$ is a constant. We also show that
$trc(G)\leq 7n/4-3$ for $\delta=3$, $trc(G)\leq8n/5-13/5$ for
$\delta=4$ and $trc(G)\leq3n/2-3$ for $\delta=5$. Furthermore, an
example shows that our bound can be seen tight up to additive
factors when $\delta\geq\sqrt{n-2}-1$.

{\flushleft\bf Keywords}: total-colored graph; total rainbow
connection; minimum degree; 2-step dominating set.

{\flushleft\bf AMS subject classification 2010}: 05C15, 05C40,
05C69, 05D40.
\end{abstract}

\section{Introduction}

In this paper, all graphs considered are simple, finite and
undirected. We refer to book \cite{2} for undefined notation and
terminology in graph theory. Let $G$ be a connected graph on $n$
vertices with minimum degree $\delta$. A path in an edge-colored
graph $G$ is a {\it rainbow path} if its edges have different
colors. An edge-colored graph $G$ is {\it rainbow connected} if any
two vertices of $G$ are connected by a rainbow path of $G$. The {\it
rainbow connection number}, denoted by $rc(G)$, is defined as the
smallest number of colors required to make $G$ rainbow connected.
Chartrand et al. \cite{5} introduced these concepts. Notice that
$rc(G)=1$ if and only if $G$ is a complete graph and that
$rc(G)=n-1$ if and only if $G$ is a tree. Moreover, $diam(G)\leq
rc(G)\leq n-1$. A lot of results on the rainbow connection have been
obtained; see \cite{LSS, LSu}.

From \cite{CFMY} we know that to compute the number $rc(G)$ of a
connected graph $G$ is NP-hard. So, to find good upper bounds is an
interesting problem. Krivelevich and Yuster \cite{8} obtained that
$rc(G)\leq20n/{\delta}$. Caro et al. \cite{3} obtained that
$rc(G)\leq\frac{\ln\delta}{\delta}n(1+o_{\delta}(1))$. Finally,
Chandran et al. \cite{4} got the following benchmark result.
\begin{thm}\label{thm 1}\cite{4} For every connected graph $G$
of order $n$ and minimum degree $\delta$, $rc(G)\leq 3n/{(\delta +1)}+3$.
\end{thm}

The concept of rainbow vertex-connection was introduced by
Krivelevich and Yuster in \cite{8}. A path in a vertex-colored graph
$G$ is a {\it vertex-rainbow path} if its internal vertices have
different colors. A vertex-colored graph $G$ is {\it rainbow
vertex-connected} if any two vertices of $G$ are connected by a
vertex-rainbow path of $G$. The {\it rainbow vertex-connection
number}, denoted by $rvc(G)$, is defined as the smallest number of
colors required to make $G$ rainbow vertex-connected. Observe that
$diam(G)-1\leq rvc(G)\leq n-2$ and that $rvc(G)=0$ if and only if
$G$ is a complete graph. The problem of determining the number
$rvc(G)$ of a connected graph $G$ is also NP-hard; see \cite{CLL,
CLS}. There are a few results about the upper bounds of the rainbow
vertex-connection number. Krivelevich and Yuster \cite{8} proved
that $rvc(G)\leq 11n/{\delta}$. Li and Shi \cite{9} improved this
bound and showed the following results.
\begin{thm}\label{thm 2}\cite{9} For a connected graph $G$ of
order $n$ and minimum degree $\delta$, $rvc(G)\leq 3n/4-2$ for
$\delta=3$, $rvc(G)\leq 3n/5-8/5$ for $\delta=4$ and $rvc(G)\leq
n/2-2$ for $\delta=5$. For sufficiently large $\delta$, $rvc(G)\leq
(b \ln\delta)n/\delta$, where $b$ is any constant exceeding 2.5.
\end{thm}
\begin{thm}\label{thm 3}\cite{9} A connected graph $G$ of order $n$
with minimum degree $\delta$ has $rvc(G)\leq 3n/(\delta +1)+5$ for
$\delta\geq \sqrt{n-1}-1$ and $n\geq 290$, while $rvc(G)\leq
4n/(\delta+1)+5$ for $16\leq \delta\leq\sqrt{n-1}-2$ and $rvc(G)\leq
4n/(\delta +1)+C(\delta)$ for $6\leq \delta\leq15$, where
$C(\delta)=e^{\frac{3\log({\delta}^3+2{\delta}^2+3)-3(\log3-1)}{\delta-3}}-2$.
\end{thm}

Recently, Liu et al. \cite{10} proposed the concept of total rainbow
connection. A {\it total-colored graph} is a graph $G$ such that
both all edges and all vertices of $G$ are colored. A path in a
total-colored graph $G$ is a {\it total rainbow path} if its edges
and internal vertices have distinct colors. A total-colored graph
$G$ is {\it total-rainbow connected} if any two vertices of $G$ are
connected by a total rainbow path of $G$. The total rainbow
connection number, denoted by $trc(G)$, is defined as the smallest
number of colors required to make $G$ total-rainbow connected. It is
easy to observe that $trc(G)=1$ if and only if $G$ is a complete
graph. Moreover, $2diam(G)-1\leq trc(G)\leq 2n-3$. The following
proposition gives an upper bound of the total rainbow connection
number.
\begin{pro}\label{prop 1}\cite{10} Let $G$ be a connected graph on $n$
vertices and $q$ vertices having degree at least $2$. Then,
$trc(G)\leq n-1+q$, with equality if and only if $G$ is a tree.
\end{pro}

From Theorem \ref{thm 1} and \ref{thm 3}, one can see that $rc(G)$
and $rvc(G)$ are bounded by a function of the minimum degree
$\delta$, and that when $\delta$ is in linear with $n$, then both
$rc(G)$ and $rvc(G)$ are some constants. In this paper, we will use
the same idea in \cite{9} to obtain upper bounds for the number
$trc(G)$, which are also functions of $\delta$ and imply that when
$\delta$ is in linear with $n$, then $trc(G)$ is a constant.

\section{Main results}

Let $G$ be a connected graph on $n$ vertices with minimum degree
$\delta$. Denote by $Leaf(G)$ the maximum number of leaves in any
spanning tree of $G$. If $\delta=3$, then $Leaf(G)\geq n/4+2$ which
was proved by Linial and Sturtevant (unpublished). In \cite{6, 7},
it was proved that $Leaf(G)\geq 2n/5+8/5$ for $\delta=4$. Moreover,
Griggs and Wu \cite{6} showed that if $\delta=5$, then $Leaf(G)\geq
n/2+2$. For sufficiently large $\delta$,
$Leaf(G)\geq(1-b\ln\delta/\delta)n$, where $b$ is any constant
exceeding 2.5, which was proved in \cite{7}. Thus, we can get the
following results.
\begin{thm}\label{thm 4} For a connected graph $G$ of order $n$ with
minimum degree $\delta$, $trc(G)\leq 7n/4-3$ for $\delta=3$,
$trc(G)\leq8n/5-13/5$ for $\delta=4$ and $trc(G)\leq3n/2-3$
for $\delta=5$. For sufficiently large $\delta$,
$trc(G)\leq(1+b\ln\delta/{\delta})n-1$, where $b$ is any
constant exceeding 2.5.
\end{thm}
\begin{proof} We can choose a spanning tree $T$ with the most leaves.
Denote $\ell$ the maximum number of leaves. Then color all non-leaf
vertices and all edges of $T$ with $2n-\ell-1$ colors, each
receiving a distinct color. Hence, $trc(G)\leq 2n-\ell-1$.
\end{proof}
\begin{thm}\label{thm 5} For a connected graph $G$ of order
$n$ with minimum degree $\delta$, $trc(G)\leq6n/{(\delta+1)}+28$ for
$\delta\geq\sqrt{n-2}-1$ and $n\geq 291$, while
$trc(G)\leq7n/{(\delta+1)}+32$ for $16\leq\delta\leq\sqrt{n-2}-2$
and $trc(G)\leq7n/{(\delta+1)}+4C(\delta)+12$ for
$6\leq\delta\leq15$, where
$C(\delta)=e^{\frac{3\log({\delta}^3+2{\delta}^2+3)-3(\log3-1)}{\delta-3}}-2$.
\qed\end{thm}
\begin{remark}\label{remark 1} The same example mentioned in \cite{3}
can show that our bound is tight up to additive factors when
$\delta\geq\sqrt{n-2}-1$.
\end{remark}

In order to prove Theorem \ref{thm 5}, we need some lemmas.
\begin{lem}\label{lem1}\cite{8} If $G$ is a connected graph
of order $n$ with minimum degree $\delta$, then it has a connected
spanning subgraph with minimum degree $\delta$ and with less than
$n(\delta+1/(\delta+1))$ edges.
\end{lem}

Given a graph $G$, a set $D\subseteq V(G)$ is called a {\it $2$-step
dominating set} of $G$ if every vertex of $G$ which is not dominated
by $D$ has a neighbor that is dominated by $D$. A $2$-step
dominating set $S$ is {\it $k$-strong} if every vertex which is not
dominated by $S$ has at least $k$ neighbors that are dominated by
$S$. If $S$ induces a connected subgraph of $G$, then $S$ is called
a {\it connected $k$-strong $2$-step dominating set}. These concepts
can be found in \cite{8}.
\begin{lem}\label{lem2}\cite{9} If $G$ is a connected graph of order
$n$ with minimum degree $\delta\geq2$, then $G$ has a connected
$\delta/3$-strong 2-step dominating set $S$ whose size is at most
$3n/(\delta+1)-2$.
\end{lem}
\begin{lem}\label{lem3}\cite{1} ($Lov\acute{a}$sz Local Lemma)
Let $A_{1},A_{2},...,A_{n}$ be the events in an arbitrary
probability space. Suppose that each event $A_{i}$ is mutually
independent of a set of all the other events $A_{j}$ but at most
$d$, and that $P[A_{i}]\leq p$ for all $1\leq i\leq n$. If
$ep(d+1)<1$, then $Pr[\bigwedge_{i=1}^{n}\bar{A_{i}}]>0$.\qed
\end{lem}
\begin{figure}
\centering
\subfigure[]{
\label{Fig.sub.1}
\includegraphics[width=0.4\textwidth]{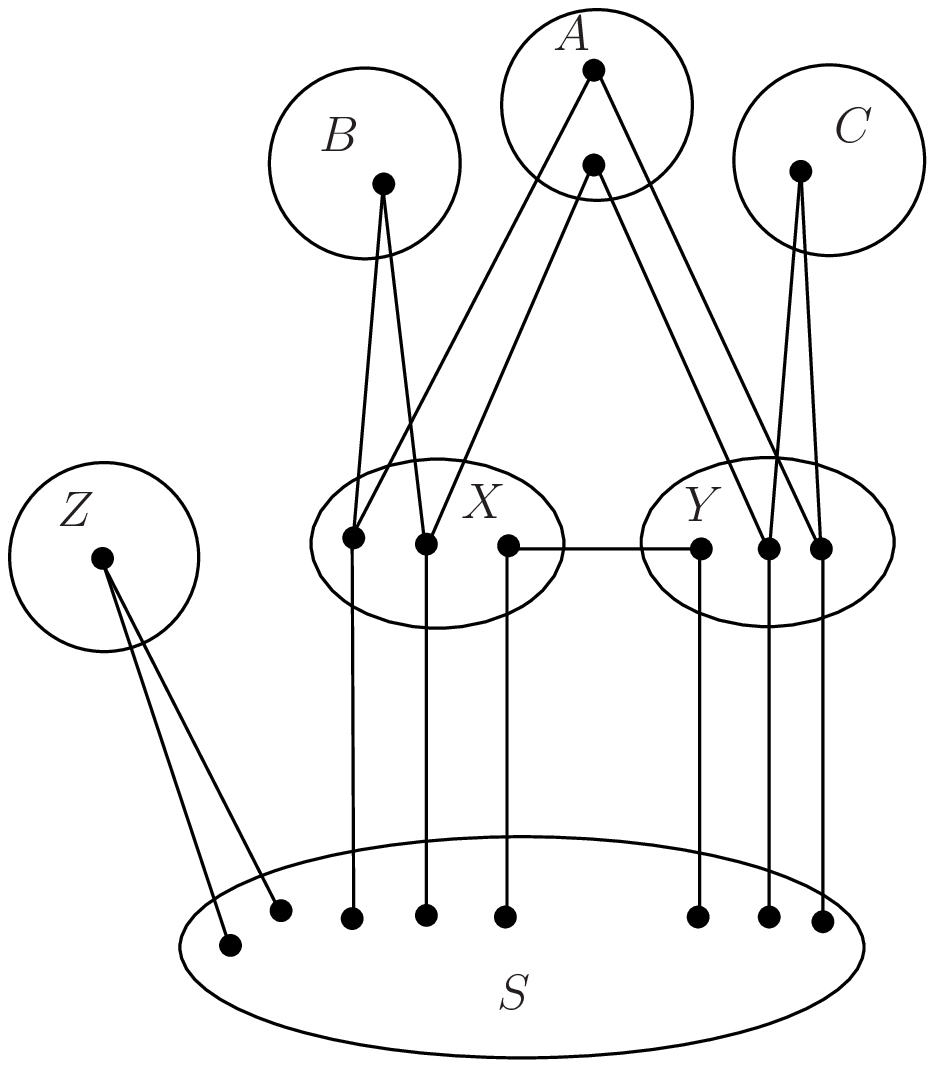}}
\subfigure[]{
\label{Fig.sub.2}
\includegraphics[width=0.45\textwidth]{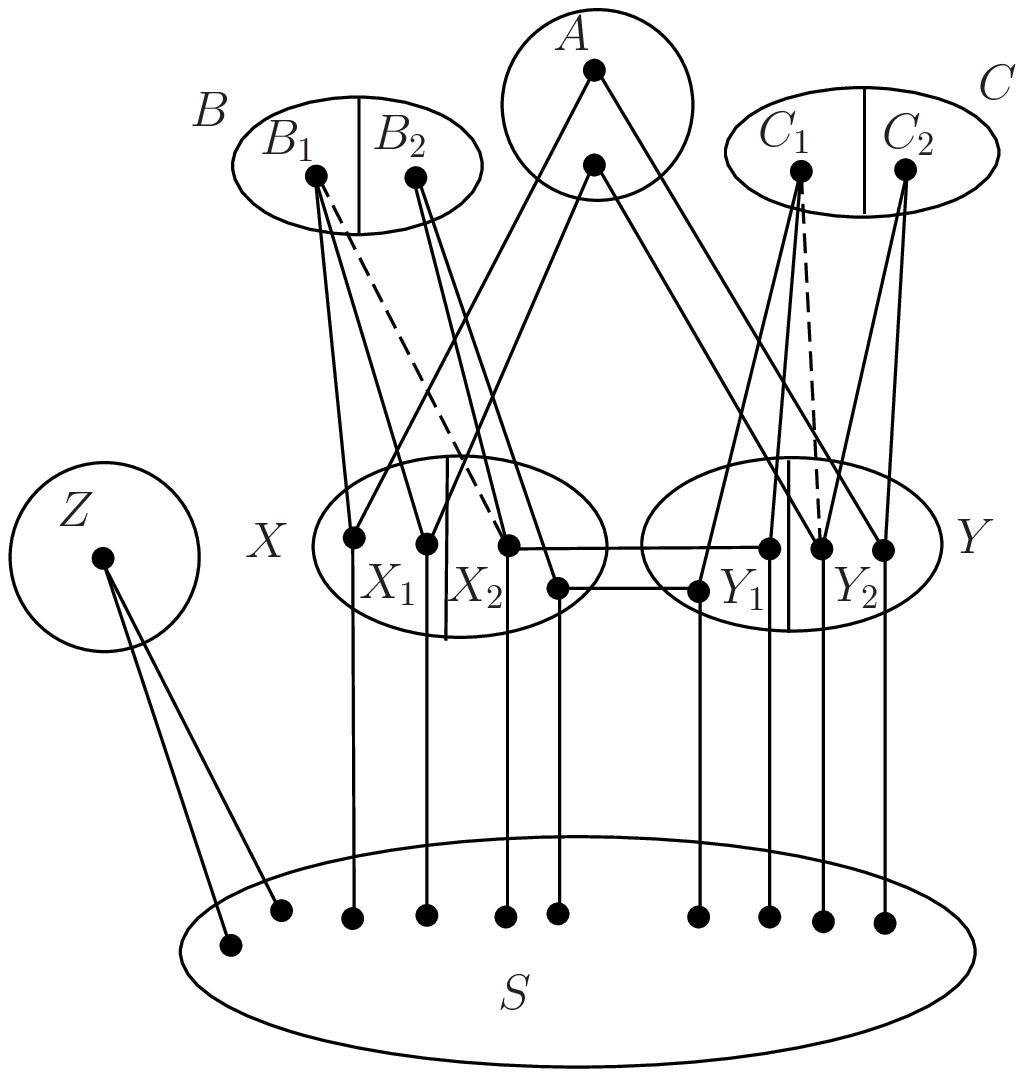}}
\caption{Illustration in the proof of Theorem 5 }\label{Fig.lable}
\end{figure}

Now we are arriving at the point to give a proof for Theorem 5.

\noindent {\bf Proof of Theorem 5:} The proof goes similarly for the
main result of \cite{9}. We are given a connected graph $G$ of order
$n$ with minimum degree $\delta$. Suppose that $G$ has less than
$n(\delta+1/(\delta+1))$ edges by Lemma \ref{lem1}. Let $S$ denote a
connected $\delta/3$-strong $2$-step dominating set of $G$. Then, we
have $|S|\leq 3n/(\delta+1)-2$ by Lemma \ref{lem2}. Let $N^{k}(S)$
denote the set of all vertices at distance exactly $k$ from $S$. We
give a partition to $N^{1}(S)$ as follows. First, let $H$ be a new
graph constructed on $N^{1}(S)$ with edge set $E(H)=\{uv: u,v\in
N^{1}(S),uv\in E(G)$ or $\exists\ w\in N^{2}(S)$ such that $uwv$ is
a path of $G\}$. Let $Z$ be the set of all isolated vertices of $H$.
Moreover, there exists a spanning forest $F$ of $V(H)\backslash Z$.
Finally, choose a bipartition defined by this forest, denoted by $X$
and $Y$. Partition $N^{2}(S)$ into three subsets: $A=\{u\in
N^{2}(S):u\in N(X)\cap N(Y)\}$, $B=\{u\in N^{2}(S):u\in
N(X)\backslash N(Y)\}$ and $C=\{u\in N^{2}(S): u\in N(Y)\backslash
N(X)\}$; see Figure 1(a).

\noindent {\bf Case 1.} $\delta\geq\sqrt{n-2}-1$.

Next we give a coloring to the edges and vertices of $G$. Let
$k=2|S|-1$ and $T$ be a spanning tree of $G[S]$. Color the edges and
vertices of $T$ with $k$ distinct colors such that $G[S]$ is total
rainbow connected. Assign every $[X, S]$ edge with color $k+1$,
every $[Y,S]$ edge with color $k+2$ and every edge in $N^{1}(S)$
with color $k+3$. Since the minimum degree $\delta\geq2$, every
vertex in $Z$ has at least two neighbors in $S$. Color one edge with
$k+1$ and all others with $k+2$. Assign every $[A,X]$ edge with
color $k+3$, every $[A, Y]$ edge with color $k+4$ and every vertex
of $A$ with color $k+5$. We assign seven new colors from
$\{i_{1},i_{2},...,i_{7}\}$ to the vertices of $X$ such that each
vertex of $X$ chooses its color randomly and independently from all
other vertices of $X$. Similarly, we assign another seven colors to
the vertices of $Y$. Assign seven colors from
$\{j_{1},j_{2},...,j_{7}\}$ to the edges between $B$ and $X$ as
follows: for every vertex $u\in B$, let $N_{X}(u)$ denote the set of
all neighbors of $u$ in $X$; for every vertex $u'\in N_{X}(u)$, if
we color $u'$ with $i_{t}\ (t\in\{1,2,...,7\})$, then color $uu'$
with $j_{t}$. In a similar way, we assign seven new colors to the
edges between $C$ and $Y$. All other edges and vertices of $G$ are
uncolored. Thus, the number of all colors we used is
$$k+33=2|S|-1+33\leq2(\frac{3n}{\delta+1}-2)-1+33=\frac{6n}{\delta+1}+28.$$

We have the following claim for any $u\in B\ (C)$.
\begin{claim}\label{claim1} For any $u\in B\ (C)$, we have a coloring
for the vertices in $X\ (Y)$ with seven colors such that there exist
two neighbors $u_{1}$ and $u_{2}$ in $N_{X}(u)\ (N_{Y}(u))$ that
receive different colors. Hence, the edges $uu_{1}$ and $uu_{2}$ are
also colored differently.
\end{claim}

Notice that for every vertex $v\in X$, $v$ has two neighbors in
$S\cup A\cup Y$. Moreover, $(\delta+1)^{2}\geq n-2$. Thus, $v$ has
less than $(\delta+1)^{2}$ neighbors in $B$. For every vertex $u\in
B$, $u$ has at least $\delta/3$ neighbors in $X$ since $S$ is a
connected $\delta/3$-strong $2$-step dominating set of $G$. Let
$A_{u}$ denote the event that $N_{X}(u)$ receives at least two
distinct colors. Fix a set $X(u)\subset N_{X}(u)$ with
$|X(u)|=\lceil\delta/3\rceil$. Let $B_{u}$ denote the event that all
vertices of $X(u)$ are colored the same. Hence, $Pr[B_{u}]\leq
7^{-\lceil\delta/3\rceil+1}$. Moreover, the event $B_{u}$ is
independent of all other events $B_{v}$ for $v\neq u$ but at most
$((\delta+1)^{2}-1)\lceil\delta/3\rceil$ of them. Since
$e\cdot7^{-\lceil\delta/3\rceil+1}(((\delta+1)^{2}-1)\lceil\delta/3\rceil+1)<1$,
for all $\delta\geq\sqrt{n-2}-1$ and $n\geq 291$, we have
$Pr[\bigwedge_{u\in B}\bar{B_{u}}]>0$ by Lemma \ref{lem3}.
Therefore, $Pr[A_{u}]>0$.

We will show that $G$ is total-rainbow connected. Take any two
vertices $u$ and $w$ in $V(G)$. If they are all in $S$, there is a
total rainbow path connecting them in $G[S]$. If one of them is in
$N^{1}(S)$, say $u$, then $u$ has a neighbor $u'$ in $S$. Thus,
$uu'Pw$ is a required path, where $P$ is a total rainbow path in
$G[S]$ connecting $u'$ and $w$. If one of them is in $X\cup Z$, say
$u$, and the other is in $Y\cup Z$, say $w$, then $u$ has a neighbor
$u'$ in $S$ and $w$ has a neighbor $w'$ in $S$. Hence, $uu'Pw'w$ is
a required path, where $P$ is a total rainbow path connecting $u'$
and $w'$ in $G[S]$. If they are all in $X$, then there exists a
$u'\in Y$ such that $u$ and $u'$ are connected by a single edge or a
total rainbow path of length two. We know that $u'$ and $w$ are
total-rainbow connected. Therefore, $u$ and $w$ are connected by a
total rainbow path. If one of them is in $A\cup B$, say $u$, and the
other is in $A\cup C$, say $w$, then $u$ has a neighbor $u'$ in $X$,
and $w$ has a neighbor $w'$ in $Y$. Thus, they are total-rainbow
connected. If they are all in $B$, by Claim 1 $u$ has two neighbors
$u_{1}$ and $u_{2}$ in $X$ such that $u_{1},\ u_{2},\ uu_{1}$ and
$uu_{2}$ are colored differently. Similarly, we also have that $w$
has two neighbors $w_{1}$ and $w_{2}$ in $X$ such that $w_{1},\
w_{2},\ ww_{1}$ and $ww_{2}$ are colored differently. Hence, $u$ and
$w$ are total-rainbow connected. We can check that $u$ and $w$ are
total-rainbow connected in all other cases.

\noindent {\bf Case 2.} $6\leq\delta\leq\sqrt{n-2}-2$.

We partition $X$ into two subsets $X_{1}$ and $X_{2}$. For any $u\in
X$, if $u$ has at least $(\delta+1)^2$ neighbors in $B$, then $u\in
X_{1}$; otherwise, $u\in X_{2}$. Similarly, we partition $Y$ onto
two subsets $Y_{1}$ and $Y_{2}$. Note that $|X_{1}\cup Y_{1}|\leq
n/(\delta+1)$ since $G$ has less than $n(1+1/(\delta+1))$ edges.
Partition $B$ into two subsets $B_{1}$ and $B_{2}$. For any $u\in
B$, if $u$ has at least one neighbor in $X_{1}$, then $u\in B_{1}$;
otherwise, $u\in B_{2}$. In a similar way, we partition $C$ into two
subsets $C_{1}$ and $C_{2}$; see Figure 1(b).

For $16\leq\delta\leq\sqrt{n-2}-2$, assume that $C(\delta)=5$; for
$6\leq\delta\leq 15$, assume that
$C(\delta)=e^{\frac{3\log({\delta}^3+2{\delta}^2+3)-3(\log3-1)}{\delta-3}}-2$.
Now we give a coloring to the edges and vertices of $G$. Let
$k=2|S|-1$ and $T$ be a spanning tree of $G[S]$. Color the edges and
vertices of $T$ with $k$ distinct colors. Assign every $[X,S]$ edge
with color $k+1$, every $[Y, S]$ edge with color $k+2$ and every
edge in $N^{1}(S)$ with color $k+3$. Since every vertex in $Z$ has
at least two neighbors in $S$, color one edge with $k+1$ and all
others with $k+2$. Assign every $[A, X]$ edge with color $k+3$,
every $[A,Y]$ edge with color $k+4$ and every vertex of $A$ with
color $k+5$. Assign distinct colors to each vertex of $X_{1}\cup
Y_{1}$ and $C(\delta)+2$ new colors from
$\{i_{1},i_{2},...,i_{C(\delta)+2}\}$ to the vertices of $X_{2}$
such that each vertex of $X_{2}$ chooses its color randomly and
independently from all other vertices of $X_{2}$. Similarly, we
assign $C(\delta)+2$ new colors to the vertices of $Y_{2}$. For
every vertex $v\in B_{1}$, if $v$ has at least two neighbors in
$X_{1}$, color one edge with $k+6$ and all others with $k+7$; if $v$
has only one neighbor in $X_{1}$, then it has another neighbor in
$X_{2}$ since $S$ is a connected $\delta/3$-strong $2$-step
dominating set. Thus, color the edge incident with $X_{1}$ with
$k+6$ and all edges incident with $X_{2}$ with $k+7$. We assign
$C(\delta)+2$ colors from $\{j_{1},j_{2},...,j_{C(\delta)+2}\}$ to
the edges between $B_{2}$ and $X_{2}$. For every vertex $u\in
B_{2}$, let $N_{X_{2}}(u)$ denote all the neighbors of $u$ in
$X_{2}$. For every vertex $u'\in N_{X_{2}}(u)$, if we color $u'$
with $i_{t}\ (t\in\{1,2,...,C(\delta)+2\})$, then color $uu'$ with
$j_{t}$. In a similar way, we assign another $C(\delta)+4$ colors to
the edges between $C$ and $Y$. All other edges and vertices of $G$
are uncolored. Hence, the number of all colors we used is
$$k+|X_{1}\cup Y_{1}|+4C(\delta)+17\leq
2(\frac{3n}{\delta+1}-2)-1+\frac{n}{\delta+1}+4C(\delta)+17
=\frac{7n}{\delta+1}+4C(\delta)+12.$$

We have the following claim for any $u\in B_{2}\ (C_{2})$.
\begin{claim}\label{claim2} For any $u\in B_{2}\ (C_{2})$,
we have a coloring for the vertices in $X_{2}\ (Y_{2})$ with
$C(\delta)+2$ colors such that there exist two neighbors $u_{1}$ and
$u_{2}$ in $N_{X_{2}}(u)\ (N_{Y_{2}}(u))$ that receive different
colors. Thus, the edges $uu_{1}$ and $uu_{2}$ are also colored
differently.
\end{claim}

Notice that every vertex $u$ of $B_{2}$ has at least $\delta/3$
neighbors in $X_{2}$ since $S$ is a connected $\delta/3$-strong
$2$-step dominating set of $G$. Let $A_{u}$ denote the event that
$N_{X_{2}}(u)$ receives at least two distinct colors. Fix a set
$X_{2}(u)\subset N_{X_{2}}(u)$ with
$|X_{2}(u)|=\lceil\delta/3\rceil$. Let $B_{u}$ denote the event that
all vertices of $X_{2}(u)$ are colored the same. Therefore,
$Pr[B_{u}]\leq (C(\delta)+2)^{-\lceil\delta/3\rceil+1}$. Moreover,
the event $B_{u}$ is independent of all other events $B_{v}$ for
$v\neq u$ but at most $((\delta+1)^{2}-1)\lceil\delta/3\rceil$ of
them. Since
$e\cdot(C(\delta)+2)^{-\lceil\delta/3\rceil+1}(((\delta+1)^{2}-1)\lceil\delta/3\rceil+1)<1$,
we have $Pr[\bigwedge_{u\in B_{2}}\bar{B_{u}}]>0$ by Lemma 3. Hence,
we have $Pr[A_{u}]>0$.

Similarly, we can check that $G$ is also total-rainbow connected.

The proof is now complete. \qed

\end{document}